\documentclass[10pt,a4paper]{amsart}

\usepackage[english]{babel}
\usepackage[T1]{fontenc}
\usepackage{textcomp}
\usepackage{amsfonts,amssymb, amsmath}
\usepackage{calrsfs}
\usepackage[dvipsnames]{xcolor}
\usepackage{enumitem}
\usepackage{ upgreek }

\makeatletter
\def\subsection{\@startsection{subsection}{3}%
  \z@{.5\linespacing\@plus.7\linespacing}{.5\linespacing}%
  {\bfseries}}
\makeatother

\newcommand{\hr}{H^2_{\mathbb{R}}}
\newcommand{\hq}{H^2_{\mathbb{H}}}

\newcommand{\hnq}{H^n_{\mathbb{H}}}

\newcommand{\CC}{\mathbb{C}}
\newcommand{\RR}{\mathbb{R}}
\newcommand{\HH}{\mathbb{H}}

\newcommand{\NN}{\mathbb{N}}

\newcommand{\VV}{\mathcal{V}}

\newcommand{\MM}{\mathcal{M}}

\newcommand{\bpm}{\left(\begin{smallmatrix}}
\newcommand{\epm}{\end{smallmatrix}\right)}

\newcommand{\bbm}{\left[\begin{smallmatrix}}
\newcommand{\ebm}{\end{smallmatrix}\right]}

\newcommand{\Int}[3]{\displaystyle{\int_{#1}^{#2} #3}}

\newcommand{\quotientgroup}[2]{\raisebox{1.6pt}{\small \newline \ensuremath{#1}}\!/\!\raisebox{-1.6pt}{\small \ensuremath{#2}}}

\newcommand{\ps}[2]{\langle #1,#2\rangle}

\newcommand{\norme}[1]{|\mspace{-3 mu}|#1|\mspace{-3 mu}|}

\newcommand{\Stab}{\text{Stab}}

\theoremstyle{remark}
\newtheorem{rmq}{Remark}[section]

\theoremstyle{definition}

\theoremstyle{plain}
\newtheorem{thmprincipal}{Theorem}

\newtheorem{thm}{Theorem}[section]

\newtheorem{cor}[thm]{Corollary}
\newtheorem{lem}[thm]{Lemma}

\title{Maximal radius of quaternionic hyperbolic manifolds}
\author{Zo\'e Philippe}

\begin{document}

\begin{abstract}
We derive an explicit lower bound on the radius of a ball embedded in a quaternionic hyperbolic manifold. 
\end{abstract}

\maketitle
\begin{center} August 20, 2015 \end{center}

\hrulefill


\section{Introduction}

It has been known since the end of the 1960's with the work of Ka\v zdan and Margulis \cite{KazdanMarg}, that any locally symmetric manifold of non-compact type contains an embedded ball of radius $r_G$ depending only on the group $G$ of isometries of its universal cover. 
Given a symmetric space $X$, denoting by $G=I(X)$ its isometry group, a lower bound on $r_{G}$ provides geometric information on any manifold obtained as a quotient of $X$: for instance, one can then deduce a lower bound on the maximal injectivity radius of any such manifold, and information about its sick-thin decomposition.

In this text we focus on the case where $X$ is the quaternionic hyperbolic space $\hnq$. We adapt techniques developed by Martin \cite{Martin1} in the real hyperbolic setting  to obtain a bound $\lambda_n$ on the maximal radius of a real hyperbolic $n$-manifold. These ideas where recently adapted to the complex hyperbolic case by Jiang, Wang and Xie \cite{XieWang}. 

The bounds given by the above-mentioned authors, and the one presented in this article, both decrease exponentially with the dimension, though the methods employed do not allow us to discuss their optimality. The description of the behaviour of the maximal radius with the dimension (can it be uniformly bounded? Could it grow with dimension?) is a matter that doesn't seem well understood yet.

On the other hand, for \emph{open} real hyperbolic manifold, Gendulphe \cite{Gendulphe} recently derived a bound for the maximal radius, which is dimension-free, and optimal in dimension $3$. His constructions greatly rely on packing theorems and cannot obviously be adapted to the case of spaces of non-constant sectional curvature.

The main result of this text is the following :

\begin{thmprincipal}
\label{ResultatPrincipal}
Let $\Gamma\subset \mathbf{Sp}(n,1)$ be a discrete, torsion-free, non-elementary subgroup acting by direct isometries on the quaternionic hyperbolic space $\hnq$. There exists a point 
$p\in\hnq$ such that, for all $\gamma\in \Gamma$,
\[
\begin{array}{ccc}
\rho(p, \gamma(p)) &\geq & \lambda_n,
\end{array}
\]
where $\lambda_n= \frac{0.05}{9^{n+1}}$. 
Any quaternionic hyperbolic manifold thus contains an embedded ball of radius $\lambda_n/2$.
\end{thmprincipal}

In this paper, a first part (section \ref{Preliminaires}) is dedicated to the introduction of our notation, some basic fact about linear algebra on $\HH$, and summarises results on the quaternionic hyperbolic space $\hnq$ and its isometries needed in the sequel (a more detailed exposition of these can be found in Kim and Parker's article \cite{KimParker}).

Martin's work crucially relies on a J\o rgensen-like inequality, established in \cite{Martin2}.
This inequality in turn depends on the explicit determination of a \emph{Zassenhauss neighbourhood} of the isometric group of the hyperbolic real space.
In \cite{FriedlandJorgensen}, Friedland and Hersonsky improved Martin's inequality slightly, 
and used this new version to deduce a better bound on the maximal radius of real hyperbolic manifolds. It is this improved inequality that Jiang, Wang and Xie use in \cite{XieWang}, and it is the one we shall use in this paper.

Section \ref{Inegalite} is devoted to the presentation of these results: first we exhibit a Zassenhauss neighbourhood of $\mathbf{PSp}(n,1)$, the direct isometry group of the quaternionic hyperbolic space. We then deduce the Martin-J\o rgensen inequality and, following Martin, a stronger inequality satisfied by the torsion-free lattices in $\mathbf{PSp}(n,1)$ (Theorem \ref{InegalitePrincipale}).

In section \ref{Lemmes}, we make explicit  the fact that when $A$ is an element of $\mathbf{PSp}(n,1)$,  $\norme{A}$ and $\norme{A-I}$ have to be small if $A$ doesn't displace enough a given point $\mathbf{o}$ of $\hnq$. We finish by combining these results and Theorem \ref{InegalitePrincipale} to reach our conclusion in the last section.

\renewcommand{\contentsname}{}\tableofcontents 

\section{Preliminaries}
\label{Preliminaires}

\subsection{Linear algebra on $\HH$}
\label{Notations}

In this text, $\HH$ denotes the algebra of Hamilton quaternions $\HH=\RR\oplus i\RR\oplus j\RR\oplus k\RR ,$
where $i, j$ and $k$ satisfy $i^2=j^2=k^2=-1$, $ij=-ji=k$, and $\HH^n$ the \emph{right} vector space of dimension $n$ over $\HH$.

A quaternion of modulus $1$ can be written $q=\cos(\theta)+\mu\sin(\theta)$ whith $\mu^2=-1$. To denote such a quaternion we shall use the more compact notation 
\[
\begin{array}{ccc}
\cos(\theta)+\sin(\theta)\mu &=& e^{\mu(\theta)}.
\end{array}
\]
This notation satisfies $e^{\mu(a)}e^{\mu(b)}=e^{\mu(a+b)}$, and in particular, $\overline{e^{\mu(a)}}=\left({e^{\mu(a)}}\right)^{-1}=e^{\mu(-a)}$. 

The capital roman letters ($A, B, \ldots$) denote matrices.

The letter $I$ denotes the identity matrix.

Given $A\in\MM_n(\HH)$, $A^*={}^t\overline{A}$ denotes it's transconjugate.

Let $A \in \MM_n(\HH)$. An \emph{eigenvector} of $A$ is a vector $\xi\in\HH^n$ such that $A\cdot\xi=\xi\lambda$ for some $\lambda\in\HH$. We call such a $\lambda$ a \emph{right eigenvalue} for $A$, or shortly an \emph{eigenvalue} of $A$. Observe that given $\xi$ and $\lambda$ as above, 
\[
\begin{array}{cc}
A\cdot(\xi q)=\xi q (q^{-1}\lambda q), & \forall q \in \HH, q\neq 0.
\end{array}
\]

Consequently, if $\lambda$ isn't real, one gets an infinity of eigenvalues of $A$, given by the conjugacy class of $\lambda$:
\[
\begin{array}{ccccc}
O(\lambda)&=& \left\{q^{-1}\lambda q,~q\in\HH, q\neq 0\right\}
&=&\left\{\bar{w}\lambda w, ~w\in\HH, |w|=1\right\}.
\end{array}
\]

Let $\CC=\RR[i]\subset \HH$, and $\CC^+$ be the set of elements of  $\CC$ with positive imaginary part. For all non-real quaternion $z$, the conjugacy class of $z$ has a unique representative in $\CC^+$. 

A \emph{complex eigenvalue} of $A$ is a right eigenvalue of $A$ belonging to $\CC^+$
(the set of complex eigenvalues of $A$ thus bijectively corresponds to the set of conjugacy class of non-real eigenvalues of $A$).

The \emph{spectrum} of a matrix $A\in\MM_n(\HH)$, $\sigma(A)$, is the set of conjugacy class of eigenvalues of $A$. With our terminology, it bijectively corresponds to the set of real and complex eigenvalues of $A$. The \emph{spectral radius} of $A$ is the number
\[
\begin{array}{ccc}
r_\sigma(A) &=& 
\underset{\lambda\in\sigma(A)}{\max} |\lambda|.
\end{array}
\]

The norm $\norme{\cdot}$ denotes the \emph{spectral norm} on $\MM_n(\HH)$:
\[
\begin{array}{ccc}
\norme{A} &=& \displaystyle{\sqrt{r_\sigma(A^* A)}}.
\end{array}
\]

$\mathbf{Sp}(n)$ denotes the group of \emph{unitary matrices} of $\MM_n(\HH)$:
\[
\begin{array}{ccc}
\mathbf{Sp}(n) &=& \left\{ A\in \MM_n(\HH), A^*A = I \right\}.
\end{array}
\]

In this setting, one can formulate the following spectral theorem: \emph{if $A$ is a unitary matrix, there exists a unitary matrix $U$ such that $U^*AU$ is a diagonal matrix with diagonal elements in $\CC^+$.}
(The reader may find more details on linear algebra on $\HH$ in Zhang's survey \cite{Zhang} for example, or -- regarding the spectral theory more specifically -- in Farenick and Pidkowich's paper \cite{TheoremeSpectralQuaternions}).

Let us also make a remark which shall prove useful in the later:
\begin{rmq}
\label{normeUnitaire}
if $U\in \mathbf{Sp}(n)$, $\norme{UAU^{-1}}=\norme{A}$.
\end{rmq}

The brackets $\ps{\cdot}{\cdot}$ denote an hermitian form of signature $(n,1)$ on $\HH^{n+1}$.

$\mathbf{Sp}(n,1)$ denotes the subgroup of $GL_{n+1}(\HH)$ formed by the matrices -- acting on $\HH^{n+1}$ on the left -- preserving $\ps{\cdot}{\cdot}$.

The lower roman letters ($f,g,h\ldots$) denote the isometries of $\hnq$.

\subsection{Quaternionic hyperbolic space and its isometries}

\subsubsection{The half-space model}

Let $\HH^{n,1}$ denote the quaternionic vectorial space of dimension $n+1$ $\HH^{n+1}$ endowed with an hermitian form of signature $(n,1)$. The \emph{quaternionic hyperbolic space} $\hnq$ is the grassmannian of negative lines with respect to such a form.
Precisely, we consider the sets $V_-$ and $V_0$ of negative and null vectors:
\[
\begin{array}{ccc}
V_-&=&\left\{Z\in H^{n,1},~~ \ps{Z}{Z}<0 \right\} ; \\
V_0&=&\left\{Z\in H^{n,1},~~ \ps{Z}{Z}=0 \right\} ;\\
\end{array}
\]
denote by $P$ the usual projection from $\HH^{n+1}$ onto $P^n(\HH)$, and define
\[
\begin{array}{ccccc}
\hnq&=&P(V_-) &\text{ and }\\
\partial\hnq&=& P(V_0).
\end{array}
\]
To a choice of form corresponds a choice of model for $\hnq$. In this text, we will mainly work in the \emph{half-space model}. This is the model given by the form
\[
\begin{array}{rcl}
\ps{Z}{W}&=&W^* J Z, ~~ J=\begin{bmatrix} 0 & 0 &1 \\0 & I_{n-1} & 0\\ 1&0&0\end{bmatrix}\\
&=&\overline{w_1}z_{n+1}+ \overline{w_2}z_2+\ldots+\overline{w_n}z_n+\overline{w_{n+1}}z_{1},
\end{array}
\]
where $Z$ and $W$ are two column vectors of $\HH^{n,1}$.

In our setting we thus have $P(V_-)=P\left(\left\{Z\in\HH^{n,1},~~ 2\Re(\overline{z_{n+1}}z_1)+|z_2|^2+\ldots+|z_n|^2<0 \right\}\right)$, \emph{i.e.}, in the chart $\{z_{n+1}=1\}$,
\[
\begin{array}{ccc}
\hnq &=& \left\{2\Re(z_1)+|z_2|^2+\ldots+|z_n|^2<0 \right\}.
\end{array}
\]
The boundary consists of the points
\[
\begin{array}{ccc}
Z={}^t\begin{bmatrix} z_1 & \hdots &z_n & 1\end{bmatrix}, & 2\Re(z_1)+|z_2|^2+\ldots+|z_n|^2 =0.
\end{array}
\]
together with a distinguished point at infinity
$q_\infty ={}^t\begin{bmatrix} 1&0&\hdots & 0 \end{bmatrix}$ (the unique point of $P(V_0)$ not contained in the chart $\{z_{n+1}=1\}$).

We define the \emph{horospherical height} of a point $Z\in\hnq$:
\[
\begin{array}{ccc}
u_Z &=& -(2\Re(z_1)+|z_2|^2+\ldots+|z_n|^2),
\end{array}
\]
and then the \emph{horospherical coordinates} of a point
$Z={}^t\begin{bmatrix} z_1& \hdots &z_n&1 \end{bmatrix}\in\hnq$:
\[
\begin{array}{ccc}
(\xi_Z, v_Z, u_Z)  &=& ((z_2,\ldots, z_n), 2\Im(z_1), -(2\Re(z_1)+|z_2|^2+\ldots+|z_n|^2)).
\end{array}
\]
These coordinates may be thought of as a generalization of the cartesian coordinates on $\hr$.

The \emph{vertical geodesics} are the lines $\left\{(\xi_0, v_0, u), u\in\RR^+ \right\}$ joining a point $(\xi_0, v_0,0)$ on the boundary to $q_\infty$. In particular, we will denote by $(0, \infty)$ the vertical geodesic $\left\{(0, 0, u), u\in\RR^+ \right\}$ joining $(0,0,0)$ and $q_\infty$.

We identify an origin in $\hnq$, namely the point $\mathbf{o}={}^t\begin{bmatrix} -1 & 0 &\hdots&0& 1\end{bmatrix}$, or $(0,0,2)$ in horospherical coordinates. It belongs to the vertical geodesic $(0, \infty)$.

\begin{rmq}
Another classical model is the \emph{ball model}, that comes with the choice of the form
$J_1=\begin{bmatrix} I_{n}&0 \\ 0& -1 \end{bmatrix}$. The Cayley transform from one model to the other is given by the change of basis from $J$ to $J_1$, namely the unitary matrix
\[
\begin{array}{ccc}
C &=& \begin{bmatrix} \frac{\sqrt{2}}{2} & 0 & \frac{\sqrt{2}}{2} 
\\ 0 & I_{n-1} &0 \\ 
\frac{\sqrt{2}}{2}& 0 &-\frac{\sqrt{2}}{2}
\end{bmatrix}.
\end{array}
\]
In this model, the hyperbolic space
\[
\begin{array}{ccccl}
\hnq &=& P(V_-) & = & 
P\left(\left\{Z\in\HH^{n,1},~~ |z_1|^2+|z_2|^2+\ldots+|z_n|^2 - |z_{n+1}|^2 <0 \right\}\right)\\
 &&& \simeq& \left\{|z_1|^2+|z_2|^2+\ldots+|z_n|^2<1 \right\}
\end{array}
\]
is identified with the unit ball in $\HH^n$, and the origin $\mathbf{o}$ of the half-space model is carried by the Cayley transform onto the origin $0$ of the ball (the point 
${}^t\begin{bmatrix} 0 & \hdots & 0 &1 \end{bmatrix}$ in inhomogenous coordinates). 

We shall use this model when describing the maximal compact subgroup of the isometry group of $\hnq$, that is the stabiliser of a point in $\hnq$. The computations will prove to be more elegant in this setting. However, this concerns only two small parts of our text ---the description of the elliptic elements in the next paragraph, and the proof of Lemma \ref{approximation}--- so unless otherwise explicitly stated, the reader should always think that we are working in the half-space model.
\end{rmq}

\subsubsection{Classification of the isometries}

We shall now present a couple of facts regarding the isometries of $\hnq$ that will be needed in the reminder of the text. A more detailed account can be found in the article of Kim and Parker mentioned in the introduction \cite{KimParker}. 

The direct isometry group of $\hnq$ is the group 
\[
\begin{array}{ccc}
\mathbf{PSp}(n,1) &=& \quotientgroup{\mathbf{Sp}(n,1)}{\{\pm I\}}.
\end{array}
\]
As in the real and complex hyperbolic cases, these isometries can be of one of the following three kind:
\begin{enumerate}
\item \emph{loxodromic}, if they fix exactly two points in $\partial\hnq$ (and have no fixed-point in $\hnq$);
\item \emph{parabolic}, if they fix exactly one point in $\partial\hnq$ (and have no fixed-point in $\hnq$);
\item \emph{elliptic}, if they fix a point in $\hnq$.
\end{enumerate}

In the ball model, a direct computation, using the fact that we are working with elements perserving the form $J_1$, shows that elliptic elements fixing the origin $0={}^t\begin{bmatrix}  0 & \hdots & 0 &1 \end{bmatrix}$ correspond to matrices of  $\mathbf{Sp}(n,1)$ of the form
\[
\begin{array}{ccc}
A&=&\begin{bmatrix} \Theta & 0  \\0& e^{\mu(\theta)} \end{bmatrix}, ~~ \Theta\in \mathbf{Sp}(n).
\end{array}
\] 

We thus see that
\[
\begin{array}{ccc}
\Stab(0)&\simeq&P(\mathbf{Sp}(n)\times \mathbf{Sp}(1)).
\end{array}
\]

\begin{rmq}
\label{normeStabilisateur}
Elliptic elements stabilizing $0$ thus have a norm equal to $1$. These elements correspond, under the Cayley transform $C$, to those stabilising the origin $\mathbf{o}$ in the half-space model. Since $C$ is unitary, using remark \ref{normeUnitaire}, we see that elliptic elements stabilising $\mathbf{o}$ in the half-space model also have norm $1$. This will proove usefull in the sequel.
\end{rmq}

\begin{rmq}
In the real or complex hyperbolic cases, after projectivising, we can assume that an elliptic element has the form
\[
\begin{array}{ccc}
A&=&\begin{bmatrix} \Theta & 0 \\ 0 & 1 \end{bmatrix}, ~~ \Theta\in U_{n}.
\end{array}
\] 
In our case however, scalar matrices are not central -- except for $\pm I$, and we can no longer make this assumption. This fact is responsible for a slight difference between our results and their analogue in the real and complex cases (compare lemma $4.2$ of \cite{XieWang} and lemma $4.1$ of \cite{FriedlandJorgensen} with  lemma \ref{approximation}).
\end{rmq}

\subsubsection{Elementary groups of isometries}
\label{Elementary}

The \emph{limit set} of a discrete subgroup $\Gamma$ of isometries of $\hnq$ is the set of accumulation points of the orbit of an arbitrary point $x\in\hnq$, denoted by $L(\Gamma)$. A discrete group $\Gamma$ is called \emph{non-elementary} if its limit set contains strictly more than two points, \emph{elementary} otherwise.

In case $\Gamma$ is elementary, one of the three following holds (see e.g. \cite{GromovHyperbolicGroups}):
\begin{enumerate}
\item $L(\Gamma)=\emptyset$. Then $\Gamma$ is finite.
\item $L(\Gamma)=\{x_0\}$. Then every infinite order element of $\Gamma$ is parabolic with fixed point $x_0$.
\item $L(\Gamma)=\{x_0, y_0\}$. Then every infinite order element of $\Gamma$ is loxodromic with fixed points $x_0$ and $y_0$.
\end{enumerate} 
In particular, if $\Gamma$ is discrete, elementary and \emph{torsion free}, the elements of $\Gamma$ are either \emph{all} parabolic or \emph{all} loxodromic. This is the only fact about elementary groups that we need in this paper (in the proof of our main inequality, theorem \ref{InegalitePrincipale}).

\subsubsection{Formulae}

The distance in $\hnq$ can be explicitly described in terms of the hermitian structure on $\HH^{n,1}$ (see for example Chen and Greenberg's artcile, \cite{ChenGreenberg}). If $X$ and $Y$ are two points in $\hnq$ and $\tilde{X}$, $\tilde{Y}$ two corresponding vectors of $\HH^{n,1}$,

\begin{eqnarray}
\label{formuleDistance}
\cosh\left(\frac{\rho(X,Y)}{2}\right)&=&\frac{\ps{\tilde{X}}{\tilde{Y}}\ps{\tilde{Y}}{\tilde{X}}}{\ps{\tilde{X}}{\tilde{X}}\ps{\tilde{Y}}{\tilde{Y}}}.
\end{eqnarray}

\begin{rmq}
We did not choose the same normalization as Chen and Greenberg, and in their paper, the $\frac{1}{2}$ factor doesn't appear on the left side of the equation.
In our text, the metric is normalized so that the sectional curvature is $-1$ on planes contained in quaternionic lines (and is thus  globally pinched between $-1$ and $-1/4$). 
\end{rmq}

In order to obtain a lower bound on the volume of a quaternionic hyperbolic manifold, we need to be able to compute the volume of a ball of given radius. This is done in the following lemma :

\begin{lem}
\label{volume}
The volume of a ball of radius $R$ in the quaternionic hyperbolic space is
\[
\begin{array}{ccc}
\text{Vol}(B(R)) &=&  \sigma_{4n} \frac{16^{n}}{4n} \sinh^{4n}\left(\frac{R}{2}\right)
\left(
1+
\frac{2n}{2n+1}\sinh^2\left(\frac{R}{2}\right)
\right)
\end{array}
\]
where $\sigma_{4n} =\frac{\pi^{2n}}{(2n)!}$ denotes the volume of the unit ball in $\RR^{4n}\simeq\HH^n$.
\end{lem}

\begin{proof}
In \cite[\textsection 3.2]{Hsiang}, Hsiang expresses the hyperbolic metric in polar coordinates :
\[
\begin{array}{ccc}

ds^2 &=& dr^2 + (\frac{1}{2}\sinh(2r))^2d\theta_3^2 +
		 \sinh(r)^2d\theta_{4n-4}^2\\
     &=& \frac{1}{4}\left(dr^2 + \sinh(r)^2d\theta_3^2 +
         (2\sinh(\frac{r}{2}))^2 d\theta_{4n-4}^2\right),
\end{array}
\]
where $d\theta_3^2$ is the standard metric on the unit sphere $S^3\subset \RR^4$ and $d\theta_{4n-4}^2$ the standard metric on the unit sphere $S^{4n-4}\subset \RR^{4n-3}$. This metric corresponds to a sectional curvature pinched between $-4$ et $-1$. Since we chose to normalize the curvature between $-1$ et $-1/4$, the metric we need to consider is
\[
\begin{array}{ccc}
ds^2 &=& dr^2 + \sinh(r)^2d\theta_3^2 +
         (2\sinh(\frac{r}{2}))^2 d\theta_{4n-4}^2.
\end{array}
\]
The volume form is consequently given by
\[
\begin{array}{ccc}
\omega &=& 2^{4n-4}\sinh^3(r)\sinh^{4n-4}(\frac{r}{2})drd\theta_{4n-1},
\end{array}
\]
where $d\theta_{4n-1}^2$ is the standard metric on the unit sphere $S^{4n-1}\subset \RR^{4n} \simeq \HH^n$.
Therefore, letting $B_{\HH^n}$ denote the euclidian unit ball of $\HH^n$ and $\sigma_{4n}$ its volume,
\[
\begin{array}{rcl}
\text{Vol}(B(R)) &=& \Int{B(R)}{}{\omega} 
=
\Int{0}{R}{2^{4n-4}\sinh^3(r)\sinh^{4n-4}\left(\frac{r}{2}\right)dr}
\Int{B_{\HH^n}}{}{d\theta_{4n-1}}\\
&=& \sigma_{4n} 
\Int{0}{R}{2^{4n-4}2^3\cosh^3\left(\frac{r}{2}\right)
\sinh^3\left(\frac{r}{2}\right)
\sinh^{4n-4}\left(\frac{r}{2}\right)dr}\\
&=& \sigma_{4n} 
\Int{0}{R}{2^{4n-1}
\cosh^2\left(\frac{r}{2}\right)
\frac{2}{4n}\left(\frac{1}{2}\cosh\left(\frac{r}{2}\right)
4n\sinh^{4n-1}\left(\frac{r}{2}\right)\right)dr}\\
&=& \sigma_{4n} 
\frac{2^{4n}}{4n}\left(
\cosh^2(R/2)\sinh^{4n}(R/2) -
\Int{0}{R}{2\cosh\left(\frac{r}{2}\right)\sinh^{4n+1}\left(\frac{r}{2}\right)dr}
\right) \\
&=& \sigma_{4n} 
\frac{2^{4n}}{4n}\left(
\cosh^2(R/2)\sinh^{4n}(R/2) - \frac{2}{4n+2}\sinh^{4n+2}(R/2)
\right)\\
&=& \sigma_{4n} \frac{16^{n}}{4n} \sinh^{4n}\left(\frac{R}{2}\right)
\left(
\cosh^2\left(\frac{R}{2}\right) -
\frac{1}{2n+1}\sinh^2\left(\frac{R}{2}\right)
\right)\\
&=&
\sigma_{4n} \frac{16^{n}}{4n} \sinh^{4n}\left(\frac{R}{2}\right)
\left(
1+
\frac{2n}{2n+1}\sinh^2\left(\frac{R}{2}\right)
\right).
\hfill \qedhere
\end{array}
\]
\end{proof}

\section{J\o rgensen-like inequality and consequences}
\label{Inegalite}

As we announced, we begin by giving a Zassenhauss neighbourhood for $\mathbf{Sp}(n,1)$, that is a neigbourhood of the identity in $\mathbf{Sp}(n,1)$ such that any discrete subgroup of $\mathbf{Sp}(n,1)$ generated by elements of this neighbourhood is nilpotent.

\begin{thm}
\label{Voisinage}
$\Omega = B(I, \tau)$ is a Zassenhauss neighbourhood for $\mathbf{Sp}(n,1)$, where $\tau\simeq 0.2971..$ is the positive root of the equation $2\tau(1+\tau)^2=1$.
\end{thm}

This result, with a slightly less good bound, was established by Martin in \cite{Martin2}: he obtained the Zassenhauss neighbourhood $\Omega_{O^+(1,n)}=$\mbox{$B(I, 2-\sqrt{3})$.} It was then improved and generalized by Friedland and Hersonsky in  \cite{FriedlandJorgensen} who obtained $\Omega_G=B(I, \tau)$ for a large class of Lie groups $G$.
Friedland and Hersonsky's improvement comes from an elementary remark which in our setting can be stated in this way:
\[
\begin{array}{cc}
\text{for } A\in \mathbf{Sp}(n,1), & \norme{A^{-1}}=\norme{A}.
\end{array}
\]
We give the proof of Theorem \ref{Voisinage}: it does not excessively increase the length of our text, and it reveals a crucial inequality (inequality \eqref{normes}) which we shall constantly reuse. 

\begin{proof}
Let then $A$, $B$ be in
$\Omega=\left\{M\in \mathbf{Sp}(n,1), ~\norme{M-I}<\tau \right\}$. We have 
\[
\begin{array}{rcl}
[A,B]-I &=& ABA^{-1}B^{-1} -I \\
&=& (AB-BA)A^{-1}B^{-1} \\
&=& ((A-I)(B-I)-(B-I)(A-I))A^{-1}B^{-1}.
\end{array}
\]
Hence
\begin{eqnarray}
\label{normes}
\norme{[A,B]-I} 
&\leq &
2\norme{A-I}\norme{B-I}\norme{A^{-1}}\norme{B^{-1}} \\
&< &
2\tau^2(1+\tau)^2=\tau. \notag 
\end{eqnarray}

Now, if $\Gamma\subset \mathbf{Sp}(n,1)$ is a discrete subgroup, $\Gamma\cap \Omega = \left\{A_1, \ldots, A_n\right\}$ is finite, and there exists a $r<\tau<1$ such that $\norme{A-I}<r$ for all $A_i\in\Gamma\cap\Omega$. 

From the inequality \eqref{normes}, we thus have, for all elements  $A_{i_0},\ldots,A_{i_k}$ of $\Gamma\cap\Omega$, 
\[
\begin{array}{ccc}
\norme{[A_{i_1},A_{i_0}]-I} &<&2r(1+r)^2\norme{A_{i_0}-I}<r\norme{A_{i_0}-I},
\end{array}
\]
and
\[
\begin{array}{ccc}
\norme{[A_{i_k},\ldots, [A_{i_1},A_{i_0}]\ldots]-I} &<&r^k\norme{A_{i_0}-I}.
\end{array}
\]

Hence, $\Gamma$ being discrete, there exists an integer $m$ such that for all sequence $(B_k)_{k\in\NN}$ given by
\[
\begin{array}{ccc}
B_k &=& [A_{i_k},\ldots, [A_{i_1},A_{i_0}]\ldots],
\end{array}
\]
$B_j=I ~~\forall j\geq m$. The group
$\langle A_1, A_2, \ldots, A_n\rangle$ is thus nilpotent.
\end{proof}

\begin{rmq}
\label{remarque}
A discrete and non-elementary group being non-nilpotent, we immediately see that if two elements $A$ and $B$ of $\mathbf{Sp}(n,1)$ generate a discrete non-elementary subgroup, necessarily $\max\left\{\norme{A-1},\norme{B-1}\right\}\geq \tau$.
Furthermore, if one asks $\langle A, B \rangle$ to be torsion-free, $A$ must be parabolic or loxodromic, and it is easily seen that if $\langle A, B^{-1}AB \rangle$ stabilises one or two points of the boundary of $\hnq$, then so does the group $\langle A, B \rangle$. Therefore, when $\langle A, B \rangle$ is discrete and torsion-free, if $\langle A, B^{-1}AB \rangle=\langle A, [A,B]\rangle$ is elementary, so is $\langle A, B \rangle$. Theorem \ref{Voisinage} thus has an ---almost--- immediate corollary:
\end{rmq}

\begin{cor}
\label{corvoisinage}
Let $\Gamma\subset \mathbf{Sp}(n,1)$ be a discrete, torsion-free subgroup, and $A$ and $B$ be two elements of $\Gamma$. We have the following alternative:
\begin{enumerate}
\item Either $A$ and $B$ generate an elementary subgroup of $\Gamma$;
\item Or
\[
\begin{array}{ccc}
\max\left\{\norme{A-I},\norme{B-I}\right\} &\geq & \tau\\
\end{array}
\]
and
\[
\begin{array}{ccc}
\max\left\{\norme{A-I},\norme{[A,B]-I}\right\}&\geq & \tau.
\end{array}
\]
\end{enumerate}
\end{cor}

We are now ready to established a J\o rgensen-like inequality. This inequality is originally due to Martin, in \cite{Martin1}. We state it here in it's improved version, as derived by Friedland and Hersonsky in \cite{FriedlandJorgensen}.

\begin{cor}[J\o rgensen-Martin's inequality]
\label{Jorgensen}
Let $\Gamma\subset \mathbf{Sp}(n,1)$ be a discrete torsion-free subgroup and $A$ and $B$ be two elements of $\Gamma$. Then, either $A$ and $B$ generate an elementary subgroup of $\Gamma$, or 
\[
\begin{array}{ccc}
\max\left\{\norme{B}\norme{B-I},\norme{A}\norme{A-I}\right\} & \geq & \omega
\end{array}
\]
where $\omega = \frac{1}{2}\tau^{1/2}\simeq 0.3854..$ is the positive root of the equation $2\omega(2\omega^2+1)=1$.
\end{cor}

\begin{proof}
Suppose that $A$ and $B$ are two elements of $\Gamma$ that do not generate an elementary subgroup of $\Gamma$ and such that 
\[
\begin{array}{ccc}
\norme{A}\norme{A-I} < \omega & \text{ and } & 
\norme{B}\norme{B-I} < \omega.
\end{array}
\]
Using inequality \eqref{normes} derived in the proof of Theorem \ref{Voisinage}, we get
\[
\begin{array}{ccccc}
\norme{[A,B]-I} & \leq & 2\norme{A}\norme{A-I}\norme{B}\norme{B-I} & < & 2\omega^2=\tau.
\end{array}
\]
Next, since $\langle [A,B],A\rangle$ cannot be elementary (see remark \ref{remarque}), by Corollary \ref{corvoisinage} we must have
\[
\begin{array}{ccc}
\norme{[A,[A,B]]-1}& \geq & \tau.
\end{array}
\]
Therefore, using inequality \eqref{normes} again, 
\[
\begin{array}{ccc}
2\norme{A}\norme{A-1}\norme{[A,B]}\norme{[A,B]-1} & \geq & \tau,
\end{array}
\]
so
\[
\begin{array}{ccc}
2\omega\norme{[A,B]} &\geq & 1.
\end{array}
\]
But also
\[
\begin{array}{ccccc}
\norme{[A,B]} & \leq & 1+\norme{[A,B]-1} & < & 1+\tau=\frac{1}{2\omega}
\end{array}
\]
and a contradiction.
\end{proof}

\begin{rmq}
Friedland and Hersonsky's improvement is an immediate consequence of their bettering of the Zassenhauss' neighbourhood. Martin considers the neighbourhood $B(1,2-\sqrt{3})$ and obtains the bound $\frac{1}{2}(2-\sqrt{3})^{1/2}$.
\end{rmq}

The main result of this section is the following:

\begin{thm}
\label{InegalitePrincipale}
Let $\Gamma$ be a discrete, torsion-free, non-elementary subgroup of $\mathbf{Sp}(n,1)$. There exists an \mbox{$H \in \mathbf{Sp}(n,1)$} such that
\begin{eqnarray}
\label{InegalitePrincipaleEq}
\norme{A}\norme{A-1}\geq \omega & \text{for all }A\in H \Gamma H^{-1}.
\end{eqnarray}
\end{thm}

\begin{proof}
Let us assume, without loss of generality, that no element of $\Gamma$ fixes $q_\infty$ or $0$ (the point of $\partial\hnq$ with horospherical coordinates $(0,0,0)$). 
Denote by $h_t$ the loxodromic flow from $0$ to $q_\infty$, and by $H_t$ the corresponding elements of $\mathbf{Sp}(n,1)$. 
$h_t$ converges to $q_\infty$ locally uniformly on 
$\overline{\hnq}\setminus\{0,q_\infty\}$
 as $t$ goes to $+\infty$, and 
 $h_t^{-1}=h_{-t}$ converges to $0$ locally uniformly on 
$\overline{\hnq}\setminus\{0,q_\infty\}$. 

Suppose that there is no $t\in\RR$ for which $H_t\Gamma H_t^{-1}$
satisfies \eqref{InegalitePrincipaleEq}. 

Firstly, remark that for a fixed element $A$ of $\mathbf{Sp}(n,1)$, $\norme{H_t A H_t^{-1}}$ goes to infinity as $t$ does. Indeed, denote by $\gamma$ the isometry correponding to $A$. By assumption, $\gamma$ does not fix $0$, so
\[
\gamma h_t^{-1}(o) \underset{t\to+\infty}{\longrightarrow} \gamma(0)\in\partial\hq-\{0\}.
\]
Consequently, the convergence being locally uniform,
\[
h_t\gamma h_t^{-1}(o) \underset{t\to+\infty}{\longrightarrow} q_\infty
\]
and
\[
\norme{H_tAH_t^{-1}} \underset{t\to+\infty}{\longrightarrow} \infty.
\]

Naturally, a similar argument using the fact that $\gamma$ does not fix $q_\infty$ shows that $\norme{H_tAH_t^{-1}}$ goes to infinity as $t$ goes to $-\infty$.

Next, we exhibit a sequence $t_i$ going to infinity and a injective sequence of elements $A_i$ of $\Gamma$ such that
\begin{eqnarray}
\label{inegaliteThm}
\norme{H_{t_i} A_i H_{t_i}^{-1}}\norme{H_{t_i} A_i H_{t_i}^{-1} - I} & < & \omega
\end{eqnarray} 
and
\begin{eqnarray}
\label{inegaliteThm2}
\norme{H_{t_i} A_{i+1} H_{t_i}^{-1}}\norme{H_{t_i} A_{i+1} H_{t_i}^{-1} - I} & < & \omega.
\end{eqnarray}
To make the notation less cluttered, for $t\in\RR$ and $A\in\Gamma$, we put
\[
N(t,A)=\norme{H_tAH_t^{-1}}\norme{H_tAH_t^{-1}-I}.
\]
We are thus looking for two sequences satisfying $N(t_i,A_i)<\omega$ and $N(t_i,A_{i+1})<\omega$. To achieve that, for any element $A$ of $\Gamma$ put
\[
V_A=\{t\in\RR, N(t,A)<\omega\}.
\]
Since $N(t,A)$ goes to infinty as $t$ does, if $V_A$ is non-empty, $V_A$ is a bounded open set. Further, by assumption, for all $t\in \RR$ there is an element $A\in\Gamma$ contradicting \eqref{InegalitePrincipaleEq}, and the set $\{V_A, A\in\Gamma\}$ thus forms an open cover of $\RR$ by bounded sets. 

Now, choose a locally finite open refinement of that cover, $\VV=\{V'\}$. Put $t_0=0$. $t_0$ is in some $V'\in\VV$ which is in turn contained in some $V_B$. Put $A_0=B$. 

We then construct the sequences by induction. Suppose $t_i$ and $A_i$ are constructed. We want to exhibit an element $A_{i+1}\neq A_i$ such that $t_i \in V_{A_{i+1}}$ (so that \eqref{inegaliteThm2} is satisfied). $t_i$ is in some set $V'\subset V_{A_i}$ of $\VV$. Any real close enough to the supremum of $V'$ is contained in another set $V''$ of $\VV$. 
If $V''\subset V_B$ with $B\neq A_i$, choose any such real for $t_{i+1}$ and put $A_{i+1}=B$. If this is not the case, do the same procedure with the supremum of $V''$. Since $\VV$ is locally finite, we are ensure to get out of $V_{A_i}$ after a finite number of steps. 
The sequence $t_i$ constructed in this way is stricly increasing and further, by local finiteness of $\VV$ it can not accumulate and consequently goes to infity. Also by construction, $t_i \in V_{A_i}\cap V_{A_{i+1}}$ and $A_i \neq A_{i+1}$ for all $i$.

Finally, \eqref{inegaliteThm} and \eqref{inegaliteThm2} are satisfied, and from the J\o rgensen-Martin inequality (Corollary \ref{Jorgensen}), we see that the group generated by $H_ {t_i} A_{i+1} H_{t_i}^{-1}$ and $H_{t_i} A_i H_{t_i}^{-1}$ must be elementary, hence its conjugate $\langle A_i, A_{i+1}\rangle$ must be too. 

Consequently (see \ref{Elementary}), either $A_i$ and $A_{i+1}$ are both parabolic and fix the same point $x_0$ of the boundary, or they are both loxodromic and fix the two same points $x_0$ and $y_0$ of the boundary. That being true for all $i$, we see that the $A_i$ either  are all parabolic or are all loxodromic, and have a common fix point $x_0$ on the boundary.  Further, denoting by $f_i$ the isometries corresponding to the $A_i$, we can assume -- extracting a subsequence if necessary,
\[
\begin{array}{ccc}
f_i(x) & \rightarrow & x_0
\end{array}
\]
locally uniformly on $\overline{\hnq}\setminus\{x_0\}$ if all the $f_i$ are parabolic, and locally uniformly on $\overline{\hnq}\setminus\{x_0,y_0\}$ if they are all loxodromic.
 
Now, consider the sequence $h_if_ih_i^{-1}$, with $h_i=h_{t_i}$. 
Since $0$ and $q_\infty$ are not fixed by any element of $\Gamma$, $\{0,q_\infty\}\cap\{x_0, y_0\}=\emptyset$, and, the convergence being locally uniform,
\[
\begin{array}{ccccccc}
h_if_ih_i^{-1}(0,q_\infty)&=&h_if_i(0,q_\infty)&\rightarrow&h_i(x_0)&\rightarrow q_\infty.
\end{array}
\]

But if a sequence $\{g_i\}$ of isometries of $\hnq$ satisfies $|g_i(x)-g_i(y)| \rightarrow 0$ for two distinct points $x$ and $y$ in $\overline{\hnq}$, denoting by $B_i$ the corresponding elements of $\mathbf{Sp}(n,1)$, necessarily $\norme{B_i}\rightarrow \infty$. 
However here, we see from \eqref{inegaliteThm} that $\norme{H_i A_i H_i^{-1}}$ is bounded. We thus get a contradiction, which concludes the proof of the theorem.
\end{proof}

\section{Intermediate results}
\label{Lemmes}

We now want to use Theorem \ref{InegalitePrincipale} to derive Theorem \ref{ResultatPrincipal}. To that end, given $f$ an  element of a discrete, torsion-free subgroup of isometrie of $\hnq$ and $A$ the corresponding matrix, we seek to bound from above the quantity
\[
\norme{A}\norme{A-I}
\]

by a function of the distance $\rho(\mathbf{o},f(\mathbf{o}))$, in order to obtain a contradiction if $f$ does not displace the point $\mathbf{o}$ enough.

In the rest of this section, $f$ is an isometrie of $\hnq$ and $A\in \mathbf{Sp}(n,1)$ is the corresponding matrix. We put
\[
\begin{array}{ccc}
\delta&=&\rho(\mathbf{o},f(\mathbf{o}))
\end{array}
\]
and
\[
\begin{array}{ccc}
r&=&\exp(\delta/2).
\end{array}
\]

We also put $K=\Stab(\mathbf{o})\simeq P(\mathbf{Sp}(n)\times \mathbf{Sp}(1))$. Recall, from remark \ref{normeStabilisateur}, that elements of $K$ have norm $1$.

After conjugating $A$ by an element of $\mathbf{Sp}(n,1)$, we can assume that $f$ sends $\mathbf{o}$ to a point on the vertical geodesic $(0, \infty)$, at distance $\delta$ from $\mathbf{o}$. We thus suppose that 
\[
\begin{array}{ccc}
f(\mathbf{o}) &=& \begin{bmatrix} -r^2 \\0 \\ 1 \end{bmatrix} \sim \begin{bmatrix} -r \\ 0 \\ 1/r \end{bmatrix}.
\end{array}
\]
The \emph{dilatation associated to $f$} is the loxodromic element fixing $0$ and $q_\infty$ sending $\mathbf{o}$ to $f(\mathbf{o})$, with corresponding matrix
\[
\begin{array}{ccc}
D &=& \begin{bmatrix} r & 0 &0 \\ 0& 1&0 \\ 0 & 0& 1/r \end{bmatrix}. 
\end{array}
\]
In particular, this element satisfies
\[
AD^{-1} \in K,
\]
and a immediate computation shows that
\[
\begin{array}{ccc}
\norme{D}=r &\text{ and } &\norme{D-I}=\norme{D^{-1}-I}=r-1.
\end{array}
\]

We easily bound $\norme{A}$ by above:

\begin{lem}
\label{normeA}
$
\begin{array}{ccc}
\norme{A} & \leq & r.
\end{array}
$
\end{lem}

\begin{proof}
Since $AD^{-1}\in K$, $\norme{AD^{-1}}=1$ and
\[
\norme{A} = \norme{AD^{-1}D} \leq \norme{AD^{-1}}\norme{D}=r.
\hfill \qedhere
\]
\end{proof}

Bounding $\norme{A-I}$ by above turns out to be more subtle: for some given element in $\mathbf{Sp}(n,1)$, it is not \emph{a priori} clear weather it is close to the identity or not.
For an element $R$ of $K$ however, either $R$ is of finite order, or it is an \emph{irrational rotation}: it is therefore possible to approach $I$ arbitrarily close by some power of $R$, and this is what we make explicit in lemma  \ref{approximation}.
We then use the triangular inequality to bound $\norme{A-I}$ ---actually $\norme{A^q-I}$--- from above:  
\begin{eqnarray}
\label{InegaliteTriangulaire}
\norme{A^q-I} &\leq & \norme{A^q -R^q} +\norme{R^q-I}.
\end{eqnarray}

The following lemma gives a bound for the first part of the right side of this expression:

\begin{lem}
\label{distanceorthogonale}
There exists an element $R$ of $K$ such that 
\[
\begin{array}{ccc}
\norme{A^q-R^q} & \leq & r(r^q-1).
\end{array}
\]
\end{lem}

\begin{proof}
Let $R\in K$. Recall the identity
\begin{eqnarray}
\label{identite}
A^q-R^q=(A-R)R^{q-1}+R(A-R)A^{q-2}+\ldots +R^{q-1}(A-R).
\end{eqnarray}
Using the fact that $\norme{A}=r$ and that $\norme{R}=1$ we then obtain, for all $R\in K$,
\[
\begin{array}{ccc}
\norme{A^q-R^q} &\leq &\frac{r^q-1}{r-1}\norme{A-R}.
\end{array}
\]
Set $R=AD^{-1}$. Then
\[
\begin{array}{ccccc}
\norme{A-AD^{-1}} &\leq & \norme{A}\norme{1-D^{-1}} &\leq & r(r-1),
\end{array}
\]
and finally we get 
\[
\begin{array}{ccc}
\norme{A^q-(A{D}^{-1})^q} & \leq & r(r^q-1).
\end{array}
\hfill \qedhere
\]
\end{proof}

Next, we have to bound above the second part of the right side of \eqref{InegaliteTriangulaire}. We shall do so by using the Dirichlet's pigeon-hole principle, which we recall (see fro example \cite[chpt 3 \textsection 3]{Hindry}):

\begin{lem}[Dirichlet's pigeon-hole principle]
\label{dirichlet}
Given $n$ real numbers $\theta_i\in [0,1]$, $i=1,\ldots, n$, for all $Q\geq 1$, there exists an integer $q\leq Q^n$ and integers $p_i$, $i=1,2\ldots, n$ such that
\[
\begin{array}{ccc}
\left|\theta_i -\frac{p_i}{q}\right| &\leq & \frac{1}{qQ}.
\end{array}
\]
\end{lem} 

We deduce:

\begin{lem}
\label{approximation}
let $R$ be in $K$. Then, for all $Q>1$, there exists an integer $q$, $1\leq q\leq Q^{n+1}$ such that
\[
\begin{array}{ccc}
\norme{R^q-I} & \leq & \frac{\pi}{Q}.
\end{array}
\]
\end{lem}

\begin{proof}
For this proof, we place ourselves in the ball model. Recall that $K\simeq K'$, where $K'$ is the stabilizor of the origin $0$ of the ball, the isomorphism being given by the conjugation by the Cayley transform $C$ which is unitary. By remark \ref{normeUnitaire}, we thus see that proving lemma \ref{approximation} for elements of $K'$ amounts to proving it for elements of $K$.

Let then $R$ be an element of $K'$, and write 
$R=\begin{bmatrix} R' &0 \\ 0& e^{\mu_1(2\pi \theta_{n+1})}\end{bmatrix}$. Without loss of generality, we can actually assume that
\[
\begin{array}{cc}
R=\begin{bmatrix} R' & 0 \\0 & e^{i(\pi \theta_{n+1})} \end{bmatrix}, 
& R'\in \mathbf{Sp}(n), \theta_{n+1}\in[0,1].
\end{array}
\]
We diagonalize $R'$ (by the spectral theorem, see section \ref{Notations}):
\[
\begin{array}{ccc}
R'&=&P' \begin{bmatrix} e^{i(\pi \theta_{1})}& 0 & \ldots& 0\\ &&\ddots& \\  0&\ldots&0& e^{i(\pi \theta_{n})}\end{bmatrix} P'^{-1},
\end{array}
\] 
with $P'\in \mathbf{Sp}(n)$. Then  
\[
\begin{array}{ccc}
R&=&P  R_1 P^{-1},
\end{array}
\]
with
\[ 
\begin{array}{ccc}
R_1 &=& \begin{bmatrix} e^{i(\pi\theta_1)} & 0 & \ldots & 0 \\ 0& e^{i(\pi \theta_{2})}& 0 & \ldots& 0\\ &&\ddots& \\  0&\ldots&0& e^{i(\pi \theta_{n})}&0 \\ &0& \ldots & 0 & e^{i\pi(\theta_{n+1})}\end{bmatrix} 
\end{array}
\]
and
\[
\begin{array}{cc}
P=\begin{bmatrix} P'&0 \\ 0& 1\end{bmatrix} & \in K'.
\end{array}
\]

Let $Q>1$ be an integer, and let $q$, $p_i, i=1\ldots n+1$ be integer corresponding to the $\theta_i$ as in Lemma \ref{dirichlet}. 
Put
\[
\begin{array}{ccc}
B&=&P B_1  P^{-1} \in K',
\end{array}
\] 
where
\[
\begin{array}{ccc}
B_1&=&\begin{bmatrix} 
e^{i(\pi \frac{p_1}{q})}&\ldots&0\\ 
0& e^{i(\pi \frac{p_2}{q})}&\ldots&0 \\ 
&\ddots& \\ 
0&\ldots & e^{i(\pi \frac{p_n}{q})}&0\\
0&\ldots & &e^{i(\pi \frac{p_{n+1}}{q})}
\end{bmatrix}.
\end{array}
\]
Then
\[
\begin{array}{rcl}
\norme{R-B}=\norme{R_1-B_1} 
&=& \displaystyle{\sqrt{r_\sigma\left((R_1^*-B_1^*)(R_1-B_1)\right)}}\\
&=&\max \displaystyle{\sqrt{\left|e^{i(\pi \theta_i)}-e^{i(\pi \frac{p_i}{q})}\right|^2}}\\
&=&\max\displaystyle{\sqrt{|2-2\cos(\pi(\theta_i-\frac{p_i}{q})|}}\\
&=&\max\displaystyle{\sqrt{|4\sin^2(\frac{\pi}{2}(\theta_i-\frac{p_i}{q}))|}}\\
&=&2\max|\sin(\frac{\pi}{2}(\theta_i-\frac{p_i}{q}))|\\
&\leq & \pi\max |\theta_i-\frac{p_i}{q}|\leq\frac{\pi}{qQ}.
\end{array}
\]

Finally we use the identity \eqref{identite} stated in Lemma \ref{distanceorthogonale}
and the fact that 
\mbox{$\norme{R}=\norme{B}=1$} to obtain:
\[
\begin{array}{ccc}
\norme{R^q-I} &=& \norme{R^q-B^q} \\
&\leq & q\norme{R-B} \\
&\leq & \frac{\pi}{Q}.
\end{array}
\]
\end{proof}

Let us summarize the results obtained in this section:

\begin{lem}
\label{Resume}
For all $Q>1$, there exists an integer $q$, $1< q \leq Q^{n+1}$, such that
\[
\begin{array}{ccc}
\norme{A^q}\norme{A^q-I} &\leq & r^q\left(r(r^q-1)+\frac{\pi}{Q}\right).
\end{array}
\]
\end{lem}

\begin{proof}
According to Lemma \ref{normeA} $\norme{A}\leq r$ , therefore \mbox{$\norme{A^q}\leq r^q$.} 
Combining \eqref{InegaliteTriangulaire}, Lemma \ref{distanceorthogonale} and Lemma \ref{approximation} we thus obtain
\[\norme{A^q-I}\leq r(r^q-1)+\frac{\pi}{Q}.\hfill\qedhere\]
\end{proof}

\section{Conclusion}

\subsection{Proof of theorem \ref{ResultatPrincipal}}

We are now ready to give a proof of the main theorem of this paper, which we state here again for convenience:

\begin{thm}
Let $\Gamma\subset \mathbf{Sp}(n,1)$ be a discrete, torsion-free, non-elementary subgroup acting by direct isometries on the quaternionic hyperbolic space $\hnq$. There exists a point 
$p\in\hnq$ such that, for all $A \in \Gamma$, denoting by $\gamma$ the corresponding isometry,
\[
\begin{array}{ccc}
\rho(p, \gamma(p)) &\geq & \lambda_n,
\end{array}
\]
where $\lambda_n= \frac{0.05}{9^{n+1}}$. 
Every quaternionic hyperbolic manifold thus contains an embedded ball of radius $\lambda_n/2$.
\end{thm}

\begin{proof}
Firstly, from Theorem \ref{inegaliteThm}, we know that there exists an 
$H\in \mathbf{Sp}(n,1)$ such that, for all $C\in H \Gamma H^{-1}, C\neq I$,

\begin{eqnarray}
\label{final}
\norme{C}\norme{C-I}&\geq & \omega  \simeq 0.3854.. .
\end{eqnarray}

Denoting by $h$ the isometry of $\hnq$ corresponding to $H$, we shall prove the theorem with $p=h^{-1}(\mathbf{o}).$

Assume on the contrary that there is an isometry $\gamma$ not satisfying the inequality of the theorem. Denote by $A$ the corresponding matrix in $\mathbf{Sp}(n,1)$ , put \mbox{$\hat{A}=H A H^{-1}$} and $\hat{\gamma}$ the corresponding isometry. Also, put $r=\exp(\rho(\mathbf{o},\hat{\gamma}(\mathbf{o}))/2)$. 

Next, apply Lemma \ref{Resume} with $Q=9$: there exists an integer $q\leq 9^{n+1}$ such that

\[
\begin{array}{ccc}
\norme{\hat{A}^q}\norme{\hat{A}^q-I} & \leq & r^q\left(r(r^q-1)+\frac{\pi}{9}\right).
\end{array}
\]

By assumption $r<e^{\lambda/2}$, so (since $n\geq 2$)

\[
\begin{array}{cccccc}
r & < & e^{\frac{0.025}{9^n}}&\leq & e^{\frac{0.025}{9^2}} &\text{and} \\
r^q &\leq & r^{9^n} & < & e^\frac{\lambda 9^{n+1}}{2} &= e^{0.025}. 
\end{array}
\]

Consequently
\[
\begin{array}{ccccc}
\norme{\hat{A}}\norme{\hat{A}-I} &<& e^{0.025}\left(e^\frac{0.025}{9^2}(e^{0.025}-1)+\frac{\pi}{9}\right) &\simeq & 0.3838.. < 0.3854.. 
\end{array}
\]
which contradicts \eqref{final}.
\end{proof}

\begin{rmq}

As a corollary, one can bound below the volume of a quaternionic hyperbolic manifold $\quotientgroup{\hnq}{\Gamma}$ by the volume of such a ball. 
We compute the later using Lemma \ref{volume}. We thus get the bound:

\begin{cor}
\label{MajorationVolume}
Let $M$ be a quaternionic hyperbolic manifold of dimension $n$. Then
\[
\begin{array}{ccc}
\text{Vol}(M) &\geq & 
\frac{2\pi^{2n}}{(2n)!}\frac{16^{n}}{4n} \sinh^{4n}\left(\frac{0.0175}{9^{n+1}}\right).
\end{array}
\]
\end{cor}

\end{rmq}

\bibliographystyle{alpha}
\bibliography{C:/Users/imb/Dropbox/Math/MiseEnFormeLatex/Bibliographie}

\end{document}